\documentclass[reqno]{amsart}
\usepackage[T1]{fontenc}
\usepackage{amssymb, amsmath, color, textcomp, url,tikz}
\usetikzlibrary{matrix,arrows}
\usetikzlibrary{fit}
\usetikzlibrary{positioning}
\usepackage[labelformat=empty]{caption}
\usepackage{mdwlist}
\usepackage{etoolbox}

\RequirePackage{ifpdf}
\ifpdf
   \usepackage[pdftex]{hyperref}
\else
   \usepackage[hypertex]{hyperref}
\fi

\usepackage{enumerate,xspace}

\usepackage{bbm}

\theoremstyle{plain}
\newtheorem{theorem}{Theorem}[section]

\newtheorem{prop}[theorem]{Proposition}
\newtheorem{lemma}[theorem]{Lemma}

\newtheorem*{claimstar}{Claim}
\newenvironment{claimstarproof}{\noindent\textit{Proof of
    Claim.}}{\hfill\qedsymbol \tiny{ Claim}
\medskip}

\newcounter{claimCount}
\newenvironment{claim}{\medskip \vskip-1mm\noindent
  \refstepcounter{claimCount}\textbf{Claim~\arabic{claimCount}.}}{\medskip}

\newenvironment{claimproof}{\noindent\textit{Proof of
    Claim~\arabic{claimCount}.}}{\hfill\qedsymbol
  \tiny{ Claim~\arabic{claimCount}}
\medskip}

\AtBeginEnvironment{proof}{\setcounter{claimCount}{0}}

\theoremstyle{definition}
\newtheorem{remark}[theorem]{Remark}
\newtheorem{fact}[theorem]{Fact}
\newtheorem{definition}[theorem]{Definition}
\newtheorem{example}[theorem]{Example}

\newtheorem*{notation}{Notation}

\newcommand{\nc}{\newcommand}

\nc{\Z}{\mathbb{Z}}
\nc{\N}{\mathbb{N}}
\nc{\Fp}{\mathbb{F}_p}
\nc{\UU}{\mathbb{U}}

\nc\LL{\mathcal L}
\nc\LLD{\mathcal L_{D}}
\nc\Pred{\mathcal P}

\nc{\dcfp}{\mathsf{DCF}_p}
\nc{\scf}{\mathsf{SCF}}
\nc{\scfp}{\scf_p}
\nc{\scfpe}{\scf_{p,e}}
\nc{\scfpi}{\scf_{p,\infty}}

\def\cl#1{\langle#1\rangle}
\nc{\restr}[1]{\!\!\upharpoonright_{#1}} 

\nc{\Land}{\bigwedge}
\nc{\proves}{\vdash}
\renewcommand{\phi}{\varphi}

\nc\ord{\operatorname{ord}}
\nc{\dcl}{\operatorname{dcl}}
\nc{\acl}{\operatorname{acl}}
\nc{\ldim}{\operatorname{ldim}}
\nc\inv{ ^{-1}}

\nc{\ldef}[2][]{\overline{\lambda}#1(#2)\!\downarrow}

\nc{\tp}{\operatorname{tp}}
\nc{\cf}{\text{cf. }}
\nc\icl{indiscernibly closed\xspace}

\nc\Ann{\operatorname{Ann}}
\nc{\Gr}{\operatorname{Gr}}
\nc{\pk}{\operatorname{Pk}}
\nc{\pluck}{\operatorname{\lrcorner}}

\makeatletter
\newcommand{\extp}{\@ifnextchar^\@extp{\@extp^{\,}}}
\def\@extp^#1{\mathop{\Land\nolimits^{\!#1}}}
\makeatother

\def\Ind#1#2{#1\setbox0=\hbox{$#1x$}\kern\wd0\hbox to 0pt{\hss$#1\mid$\hss}
\lower.9\ht0\hbox to 0pt{\hss$#1\smile$\hss}\kern\wd0}
\def\Notind#1#2{#1\setbox0=\hbox{$#1x$}\kern\wd0\hbox to
0pt{\mathchardef\nn="0236\hss$#1\nn$\kern1.4\wd0\hss}\hbox
to 0pt{\hss$#1\mid$\hss}\lower.9\ht0
\hbox to 0pt{\hss$#1\smile$\hss}\kern\wd0}
\def\ind{\mathop{\mathpalette\Ind{}}}
\def\nind{\mathop{\mathpalette\Notind{}}}

\def\ld{\mathop{\ \ \hbox to 0pt{\hss$\mid^{\hbox to
0pt{$\scriptstyle\mathrm{ld}$\hss}}$\hss}
\lower4pt\hbox to 0pt{\hss$\smile$\hss}\ \ }}

\def\indone{\mathop{\ \ \hbox to 0pt{\hss$\mid^{\hbox to
				0pt{$\scriptstyle\mathrm{1}$\hss}}$\hss}
		\lower4pt\hbox to 0pt{\hss$\smile$\hss}\ \ }}

\begin{document}

\title{Non-forking independence in stable theories}
\date{\today}

\author{ Amador Martin-Pizarro}
\address{Mathematisches Institut,
  Albert-Ludwigs-Universit\"at Freiburg, D-79104 Freiburg, Germany}
\email{pizarro@math.uni-freiburg.de}

\begin{abstract}
 We observe that a simple condition suffices to describes non-forking independence over models in a stable theory. 
 
Under mild assumptions, this description can be extended to non-forking independence over algebraically closed subsets, without having to use the full strength of the work of the seminal work of Kim and Pillay. The results in this note (which are surely well-known among most model theorists) essentially  use that types over models in a stable theory are stationary. 
\end{abstract}

\maketitle

\section{Non-forking and stationarity}\label{S:nonforking}
Consider a first-order complete  theory $T$ in some fixed language $\LL$ and work in a sufficiently saturated model $\mathbb U$ of $T$. All tuples and subsets are taken inside the real sort of $\UU$, unless explicitly stated.

\begin{definition}\label{D:Indep_properties}\textup{(}\cf \cite[Definition 1.1]{hA05}\textup{)}~
	An \emph{(abstract) independence relation} is a ternary relation $\indone$ between triples of subsets of $\UU$. The set $A$ is \emph{$\indone$-independent from $B$ over} $C\subset A\cap B$, denoted by $A\indone_C B$, if the triple $(A, B, C)$ belongs to the relation determined by the independence relation $\indone$. In case $C$ is not a subset of $A\cap B$, the notation $A\indone_C B$ (quite  common in model theory) should be understood as $C\cup A\indone_C C\cup B$. 
	
We now list some properties for an independence relation $\indone$ (where we implicitly assume that $C\subset A\cap B$ in the following):
	\begin{description}
	\item[\bf Invariance] Given triples $ABC$ and $A'B'C'$ with the same type (with respect to some enumeration), that is $ABC\equiv A'B'C'$,  if  $A\indone_{C} B$ holds, then so does $A'\indone_{C'} B'$.	\vskip2mm
	\item[\bf Symmetry] If $A\indone_C B$, then  $B\indone_C A$. \vskip2mm 
	\item[\bf Monotonicity] Given $C\subset B\subset D$, if $A\indone_C D$, then $A\indone_C B$  \vskip2mm
	\item[\bf Base Monotonicity] Given $C\subset B\subset D$,   if   $A\indone_C D$, then  $A\cup B\indone_B D$.  \vskip2mm
	\item[\bf Transitivity] Given $C\subset B\subset D$,   if  $A\indone_C B$ and  $A\cup B\indone_{B} D$, then  $A\indone_{C} D$.\vskip2mm
		\item[\bf Extension] Given $A$, $C$ and $B$ with $C\subset B$, there is some $A' \equiv_C	A$ with $A'\indone_C B$. \hskip2mm
	\item[\bf Finite Character] The independence $A\indone_C B$ holds if and  only if for all finite tuples $a$ in $A$ and $b$ in $B$, we have that $C\cup a\indone_C C\cup b $. \vskip2mm
	\item[\bf Local Character] For every finite tuple $a$ and every subset $B$ there is some subset $C\subset B$ with $|C|\le |T|$ such that $C\cup a\ind_CB$. \vskip2mm 
	\item[\bf Stationarity over models]  Whenever $A$ and $A'$ have the same type over the elementary substructure $M\subset B$, if both independences $M\cup A\indone_M B$ and $M\cup A'\indone_M B$ hold, then $A\equiv_B A'$. 
\end{description}
	
\end{definition}
\begin{remark}\label{R:basedness}
If the independence relation $\indone$ satisfies {\bf Invariance} and {\bf Extension}, then it also satisfies 
\begin{description}
\item[Basedness]  For every subsets $A$ and $B$, we have that $A\cup B\indone_B \acl(B)$,
\end{description}
since every automorphism fixing $B$ permutes the (model-theoretic) algebraic closure $\acl(B)$.
\end{remark}

Given a tuple $a$ and subsets $C\subset B$, we say that $a$ is non-forking independent from $B$ over $C$, denoted by $a\ind_C B$, if $\tp(a/B)$ is a non-forking extension of $\tp(a/C)$. More generally, we say that the subset $A$ is  non-forking independent from $B$ over $C$, denotes by $A\ind_C B$, if $a\ind_C B$ for every finite tuple $a$ in $A$.

Recall that the theory $T$ is \emph{stable} if there exists some infinite cardinal $\kappa$ such that $|S_n(A)|\le \kappa$ for every subset $A$ of parameters of size at most $\kappa$, where $S_n(A)$ denotes the space of $n$-types  over $A$.  By the work of Kim and Pillay \cite{KP97}, the theory $T$ is stable if and only if there exists an independence relation satisfying {\bf Invariance, Symmetry, Monotonicity, Base Monotonicity, Transitivity}, {\bf Local and Finite Character}, {\bf Extension} as well as {\bf Stationarity over models}. Moreover, such an  independence relation is unique and agrees with  non-forking independence.

\begin{notation}
{\bf From now on, we denote non-forking indepencence by $\ind$.}
\end{notation}

\begin{definition}\label{D:weak_indep}
A \emph{weak notion of independence} is an invariant independence relation $\indone$ such that $A\ind_M B$ implies $M\cup A\indone_M B$ for all subsets $A$, $B$ and  $M\subset B$ with $M$ an elementary substructure. 
\end{definition}

\begin{prop}\label{P:uniq}
	Assume that the stable theory $T$ admits a weak notion of independence  $\indone$ which satisfies {\bf Stationarity over models}. Then the relation $\indone$ agrees with non-forking independence over elementary substructures, that is, if $M\subset B$ is an elementary substructure and $A$ is a set, then \[ A\ind_M B \ \text{ if and only if } M\cup A\indone_M B.\]
		
		Furthermore, assume the weak notion of independence $\indone$ satisfies {\bf Basedness} (see Remark \ref{R:basedness}), {\bf Monotonicity, Base Monotonicity, Transitivity},  {\bf Stationarity over models} as well as   
		\begin{description}	
\item[Weak Extension] For every subsets $A$, $B$ and $C\subset B$ with $C$ algebraically closed, there is some $A'\equiv_C A$  such that   $C\cup A'\indone_C B$.
\end{description} 
In this case, the weak notion of independence $\indone$ agrees with non-forking independence over all algebraically closed subsets, that is, if $C\subset B$ is algebraically closed, then \[ A\ind_C B \ \text{ if and only if } C\cup A\indone_C B.\]
\end{prop}
Note that we do not require that the weak notion of independence $\indone$ satisfies (full) extension, symmetry nor local or finite character. 
\begin{proof}~
	
\begin{claim}\label{Beh:eq_models}
Every weak notion of  independence $\indone$ satisfying {\bf Stationarity over models} agrees with non-forking independence over elementary substructures. 
\end{claim}

\begin{claimproof}
Consider subsets $A$, $B$ and $M\subset B$ with $M$ an elementary substructure such that $M\cup A\indone_M B$. In order to show that $A\ind_M B$, choose some $A'\equiv_M A$ with $A'\ind_M B$ (since non-forking satisfies extension).  We have that $M\cup A'\indone_M B$, for $\indone$ is a weak notion of independence. Now, {\bf Stationarity over models} for $\indone$ yields that $A\equiv_B A'$, so $A\ind_M B$ by invariance of non-forking independence. 
\end{claimproof}

Assume now that the weak notion of independence $\indone$ satisfies  {\bf Basedness, Weak Extension, Monotonicity, Base Monotonicity, Transitivity} as well as {\bf Stationarity over models}. We will show that $\indone$ agrees with non-forking independence over algebraically closed subsets using the following claims. 

\begin{claim}\label{Beh:LascarPoizat}\textup{(}\cf \cite[Proposition 5.4]{LP79}\textup{)} Given subsets $A$, $B$ and $C\subset B$, we have that $A\ind_C B$ if and only if  for every elementary substructure $M\supset C$, there is some $A'\equiv_B A$ with $A' \ind_M M\cup B$. 
\end{claim}

\begin{claimproof}
	Both directions are straightforward applications of non-forking yoga:  If $A\ind_C B$, given an arbitrary elementary substructure $M\supset C$, choose $A'\equiv_B A$ with $A'\ind_B M\cup B$. Since $A'\ind_C B$ (for $A\ind_C B$), we deduce now from transitivity of non-forking independence that $A'\ind_C M\cup B$, so $A'\ind_M M\cup B$ by base monotonicity. Assume now that the right-hand side holds for all elementary substructures, and choose now such an elementary substructure $M\supset C$ with $M\ind_C B$.  \end{claimproof}

\begin{claim}\label{Beh:partone}
If $C\subset B$ is algebraically closed and $C\cup A\indone_C B$, then  $A\ind_C B$.
\end{claim}

\begin{claimproof}
Assume that $C\cup A\indone_C B$. By {\bf Basedness} and {\bf Transitivity}, we may also assume that $B$ is algebraically closed. In order to show that $A\ind_C B$, it suffices by Claim \ref{Beh:LascarPoizat} to show that 
for every elementary substructure $M\supset C$, there is some $A'\equiv_B A$ with $A' \ind_M M\cup B$.  By {\bf Weak Extension} there is some $A'\equiv_B A$ with $B\cup A'\indone_B M \cup B$. {\bf Transitivity} yields that $C\cup A'\indone_C M \cup B$ and thus $a'\indone_M M\cup B$ by {\bf Base Monotonicity}. We deduce from Claim \ref{Beh:eq_models} that $A'\ind_M M\cup B$, as desired. \end{claimproof}	
	
By Claim \ref{Beh:partone}, we need only show that 	$A\ind_C B$ implies $ C\cup A\indone_C B$ whenever $C\subset B$ is algebraically closed. Choose a sufficiently saturated elementary substructure $M$ containing $B$ with $A\ind_B M$, so $A\ind_C M$ by transitivity of non-forking. By {\bf Weak Extension} choose some $A'\equiv_C A$ with $C\cup A'\indone_C M$. Claim \ref{Beh:partone} yields that $A'\ind_C M$. Now, the type $\tp(A/C)$ need not be stationary (for do not require that the theory $T$ has weak elimination of imaginaries), yet all of its non-forking extensions to $M$ are conjugate over $C$, by saturation. Therefore, there is an automorphism of $M$ which maps $\tp(A'/M)$ to $\tp(A/M)$. By invariance of the relation $\indone$, we conclude that $C\cup A\indone_C M$ and thus $C\cup a\indone_C B$ by  {\bf  Monotonicity}, as desired. 
\end{proof}

\begin{remark}\label{R:weak_hyp}

Suppose that the weak notion of independence $\indone$ satisfies {\bf Symmetry,  Monotonicity, Base Monotonicity, Transitivity} and {\bf Stationarity over models}.  In order for  {\bf Weak extension} to hold for $\indone$, it suffices that  for all  $C\subset B$ with $C$ algebraically closed there is some elementary substructure $M$ containing $C$ with $M\indone_C B$.

Such an elementary substructure always exists over any subset $C$ (possibly non-algebraically closed) whenever  $\indone$ additionally satisfies  {\bf Basedness} as well as {\bf Finite Character} and is such that every unary type $\tp(a/B)$ with $C\cup a\nind^1_C B$ is algebraic.
\end{remark}
\begin{proof}
	In order to prove  {\bf Weak extension}, consider $A$, $B$ and $C\subset B$ with $C$ algebraically closed. By assumption, there is  some elementary structure $M$ containing $C$ with $M\indone_C B$. Now, by extension of non-forking independence, find some $A'\equiv_M A$ (and in particular $A\equiv_C A$) with $A'\ind_M M\cup B$. Since $\indone$ is a weak notion of independence, we have that $M\cup A'\indone_M M\cup B$. {\bf Symmetry} yields that $M\cup B\indone_M M\cup A'$, so  $B\indone_C M\cup A'$  by {\bf Transitivity} and hence $B\indone_C C\cup A'$  by {\bf Monotonicity}. Thus, we have that $C\cup A'\indone_C B$, as desired. 
	
	If $\indone$ additionally satisfies {\bf Basedness} and {\bf Finite Character}, it follows that $B\indone_C \acl(C)$. Zorn's lemma yields a maximal algebraically closed subset $D$ of $\UU$ containing $C$ with $B\indone_C D$, so $D\indone_C B$. In order to show that $D$ is an elementary substructure, we need only show by Tarski's test that every non-algebraic formula $\varphi(x, d)$ over $D$ has a realization $a'$ with $D\cup \{a'\}\indone_D D\cup B$  (so $D\cup B\indone_D D\cup a'$). Indeed, if $a'$ is such a realization, then $B\indone_{D\cup \{a'\}} \acl(D\cup \{a'\})$ by {\bf Basedness}, so $B\indone_C \acl(D\cup \{a'\})$ by {\bf Transitivity}, which implies that $a'$ belongs to $D$, by maximality of $D$. 
	
	If there is no  realization $a'$ in $\UU$ of $\varphi(x, d)$ which is $\indone$-independent  from $D\cup B$ over $D$, then every realization of the non-algebraic formula $\varphi(x,d)$ lies in the set $\acl(D\cup B)$, by our assumption, which gives the desired contradiction by compactness. 
	\end{proof}

Simplicity, as introduced by Shelah \cite{sS80}, is a generalization of stability, in which non-forking independence satisfies all properties mentioned in Definition \ref{D:Indep_properties} except possibly {\bf Stationary over models} which is then replaced by the so-called {\bf Independence Theorem over models}: given a type $p$ over an elementary substructure $M$ as well as two supersets $A$ and $B$ of $M$ with $A\ind_M B$, the partial type $p_1\cup p_2$ over $A\cup B$ given by two non-forking extensions $p_1$ of $p$ to $A$ and $p_2$ of $p$ to $B$ does not fork over $M$.  We have not attempted to investigate whether there is a suitable version of Proposition \ref{P:uniq} characterizing non-forking independence for simple theories. 

If the simple theory expands the theory of fields, Chatzidakis showed in \cite[Theorem 3.5]{zC99} that non-forking independence over models always implies linear as well as $p$-independence (if the characteristic is positive).  A simple inspection of her proof shows that very little of the ambient simple theory is needed, so her proof easily generalizes to Kim-independence. We would like to thank Omar Le\'on S\'anchez for pointing out  Chatzidakis's elegant proof.  
\begin{definition}\label{D:coheir_prop}
An invariant notion of notion of independence $\indone$ in a (possibly unstable)  complete theory $T$ is: 
\begin{itemize}
\item {\bf Coheir Robust} if $M\cup a\indone_M B$ whenever the type $\tp(a/B)$ over $B$  of the finite tuple $a$  is \emph{coheir} (that is, finitely satisfiable) over the elementary substructure $M\subset B$. \vskip2mm 
\item {\bf Weighted} (or $\indone$ has {\bf Weight}) if for some infinite cardinal $\lambda$ there are no elementary substructure $M$, a finite tuple $a$ and a sequence $(b_\alpha)_{\alpha<\lambda}$  with $ b_\alpha\equiv_{M\cup a} b_0$ such that $M\cup b_\alpha \indone_M M\cup (b_\beta)_{\beta<\alpha}$, yet  \[ M\cup a\nind^1_{M} M\cup b_\alpha \text{ (or equivalently } M\cup a\nind^1_{M} M\cup b_0 ),\] for all $\alpha<\lambda$.  
\end{itemize}
\end{definition}

\begin{remark}\label{R:rank_weight}
A notion of independence $\indone$  is \emph{ranked} if there is an ordinal-valued rank function $R$ on all finitary types (that is, types in finitely many variables) such that for a finite tuple $a$ and subsets $C\subset B$, we have that $R(\tp(a/B))\le R(\tp(a/C))$. Moreover, equality holds if and only if $C\cup a\indone_C B$. 

Every invariant ranked notion of independence satisfying  {\bf Symmetry} and {\bf Transitivity} is {\bf Weighted} with $\lambda=\aleph_0$:  Indeed, for every $n$ in $\N$ \[ a\nind^1_{M \cup (b_0, \ldots, b_{n-1})} M \cup (b_0, \ldots, b_n),\] and thus 
 \[ R(\tp(a/M\cup(b_0, \ldots, b_n)))<R(\tp(a/M\cup(b_0, \ldots, b_{n-1}))).\] 
\end{remark}

\begin{definition}\label{D:lindisj}
Work now inside an ambient algebraically closed field $\mathbb K$. 
\begin{enumerate}[(a)]
\item Two subfields $L$ and $M$ of $\mathbb K$ are \emph{linearly disjoint} over a common subfield $k$, denoted by $L\ind^\textrm{ld}_k M$, if  the elements $a_1,\ldots, a_n$ of $L$ are linearly independent over $M$ whenever they were so over $k$. 
\item The elements $a_1, \ldots, a_n$ of $\mathbb K$ are algebraically independent over the subfield $k$ of $K$ if $P(a_1,\ldots, a_n)\ne 0$ for every non-trivial polynomial $P(T_1,\ldots, T_n)$ with coefficients in $k$. 
\item Two subfields $L$ and $M$ of $\mathbb K$ are \emph{algebraically independent} over a common subfield $k$, denoted by $L\ind^\textrm{ACF}_k M$, if the elements $a_1,\ldots, a_n$ of $L$ are algebraically independent over $M$ whenever they were so over $k$. 
\end{enumerate}
\end{definition}
\begin{remark}\label{R:lindisj}
Linear disjointness and algebraic independence are  both  symmetric independence notions which satisfy  Monotonicity, Base Monotonicity and Transitivity (taking the fields generated by the subsets in question). Moreover, linear disjointness always implies algebraic independence, and it coincides with algebraic independence whenever one of the extensions $k\subset L$ or $k\subset M$ is \emph{regular}, that is, linearly disjoint from the algebraic closure $k^{alg}$ over $k$. 
\end{remark}

\begin{example}\label{E:alg_ind}
If the complete (possibly unstable) theory $T$ expands the theory of fields in the language of rings,  given sets $A, B$ and $C\subset A\cap B$ in some ambient model $\UU$ of $T$, set \[ A\indone_C B \text{ if and only if } \left \{ \parbox{7cm}{the subfields  $\mathbb{F}\cl{A}$ and $\mathbb F\cl{B}$ are linearly disjoint over  the subfield  $\mathbb F\cl{C}$} \right. ,\] where $\mathbb F\cl{X}$ is the field of fractions of the the substructure generated by the subset $X$.

The ternary relation $\indone$ clearly satisfies {\bf Invariance, Symmetry, Monotonicity, Transitivity} as well as {\bf Finite Character}. Let us show that it also satisfies 
{\bf Coheir Robustness} and is {\bf Weighted}. For {\bf Coheir Robustness}, using that linear disjointness is symmetric, suppose the elements $x_1,\ldots, x_n$ of $\mathbb F\cl{B}$  satisfy a non-trivial linear relation with coefficients in $\mathbb F\cl{M\cup a}$. By clearing denominators, we may assume that $x_1,\ldots, x_n$  are terms in $B$ and that the coefficients belong to the substructure generated by $M\cup a$. We can encode all of this by a formula $\varphi(x,b)$   for some tuple $b$ in $B$ such that $\varphi(a, b)$ holds in $\UU$. Now, the type $\tp(a/B)$ is coheir over $M$, so there is some tuple $m$ in $M$ such that $\varphi(m, b)$ holds. In particular, the elements $x_1,\ldots, x_n$ satisfy a non-trivial linear equation with coefficients in $M$, as desired. 

We will now show that the above relation $\indone$ is {\bf Weighted} with $\lambda=\aleph_0$ (though this notion of independence need not be ranked in general).  Assume for a contradiction that there are an elementary substructure $M$, a finite tuple $a$ and a sequence $(b_n)_{n\in \N}$ such that \[  b_n\equiv_{M\cup a} b_0 \ \text{ and } \  \mathbb F\cl{M\cup b_n}\ind^\textrm{ld}_M  \mathbb F\cl{M\cup (b_0,\ldots, b_{n-1})} \text{ for all $n$ in $\N$},\]  but \[  \mathbb F\cl{M\cup a}\nind^\textrm{ld}_{M} \mathbb F\cl{M\cup b_0} \]  Choose some finite $M$-linearly independent tuple $u$ in $\mathbb F\cl{M\cup a}$ witnessing the last dependence. By {\bf Invariance}, since $b_n\equiv_{M\cup a} b_0$, the same tuple witnesses that for every $n$ in $\N$  the fields $M(u)$ and $\mathbb F\cl{M\cup b_n}$ are not linearly disjoint over $M$. Now, using that \[  \mathbb F\cl{M\cup b_n}\ind^\textrm{ld}_M  \mathbb F\cl{M\cup (b_0,\ldots, b_{n-1})},\] we deduce by symmetry and transitivity of linear disjointness that \[ \mathbb F\cl{M\cup (b_0,\ldots, b_{n-1})}(u) \nind^\textrm{ld}_{ \mathbb F\cl{M\cup (b_0,\ldots, b_{n-1})} } \mathbb F\cl{M\cup (b_0,\ldots, b_n)}, \] which gives the desired contradiction after at most $\mathrm{ldim}_M(u)=|u|$ many steps. 
\end{example}
\begin{remark}\label{R:Example_weight}
If all models of the theory $T$ of fields have positive characteristic $p$, the  notion of independence given by linear disjointness together with $p$-independence  (see the definitions before Fact \ref{F:p-base}) will also satisfy 
 {\bf Invariance, Symmetry, Monotonicity, Transitivity,  Finite Character, Coheir Robustness} and {\bf Weight}, by a verbatim argument as before. 
 
 However, the notion of independence given in Example \ref{E:alg_ind} need not satisfy {\bf Base Montonicity}. It will be the case if $\mathbb F\cl{A\cup B}$ is contained in the compositum field $\mathbb F\cl{A}\cdot \mathbb F\cl{B}$. The latter holds (after possibly adding parameters) if every function symbol $f$ the language $\LL$ of $T$  is \emph{operator-like}, that is, working over an elementary substructure $M$ both $f(x+y)$ and $f(x\cdot y)$ belong to the compositum field $\mathbb F\cl{M\cup x}\cdot \mathbb F\cl{M\cup y}$. 
\end{remark}

\begin{lemma}\label{L:ld}\textup{(}\cf \cite[Theorem 3.5]{zC99}\textup{)} 
Given some ambient model of a complete (possibly unstable) theory $T$, consider finite tuples $a$ and $b$ as well as an an elementary substructure $M$. If $\tp(a/M\cup b)$ does not \emph{Kim-divide} over $M$, then $M\cup a\indone_M M\cup b$ for every notion of independence having {\bf  Weight} and satisfying {\bf Invariance} and {\bf Coheir Robustness}. 

In particular, in any theory of fields non-forking independence over models  implies linear disjointness and $p$-independence  of the fields generated by the corresponding substructures, as in Example \ref{E:alg_ind} and Remark \ref{R:Example_weight}. These two independence relations are thus weak notions of independence, as in Definition \ref{D:weak_indep}.
\end{lemma}

\begin{proof}
Assuming otherwise, find tuples $a$ and $b$ as well as an elementary substructure $M$ such that $\tp(a/M\cup b)$ does not Kim divide over $M$, yet  $M\cup a\nind^1_M M\cup b$. The type $q=\tp(b/M)$ is finitely satisfiable over $M$ and hence it admits a global extension $\tilde q$   which remains coheir over $M$ (so $ tilde q$ is $M$-invariant). If $\lambda$ is as in the the definition of {\bf Weight}, we now construct a Morley sequence  $(b_\alpha)_{\alpha<\lambda }$ of $\tilde q$ with $b_0=b$ and $b_{\alpha}$ realizing $\tilde{q}_{\hskip1mm \restr{M\cup  (b_\beta)_{\beta<\alpha}}}$ for $0\ne \alpha<\lambda$. {\bf Coheir robustness} yields that \[ M\cup b_\alpha\indone_M M\cup  (b_\beta)_{\beta<\alpha}.\] Now, the type $\tp(a/Mb)$ does not Kim-divide over $M$, so there exists in particular a tuple $a'$ with $a'b_\alpha\equiv_M ab$ for all $\alpha<\lambda$. We have that $b_\alpha\equiv_{M\cup a'} b_0$, but $M\cup a'\nind^1_M b_0$ by {\bf Invariance}, which contradicts {\bf Weight}. 

To prove the last assertion in the lemma, it suffices to notice that  non-forking independence over a model always implies Kim-independence \cite[Fact 7.2]{KR20}. 
\end{proof}

\section{Characterizing non-forking in some classical theories}\label{S:charact_examples}
\subsection*{The free pseudoplane}
The theory PS of the free pseudoplane   is a relational theory in the language $\LL=\{R\}$ consisting of a binary relation $R$ for the edge relation on the model of $T$, which  are exactly those graphs such that every vertex has infinitely many direct neighbors (with respect to the edge relation given by $R$) and with no closed reduced paths of length at least $2$. A  \emph{reduced path} of length $n$ is a sequence $a_0, a_1,\ldots, a_{n}$ with $a_i$ distinct from $a_{i+2}$ for all $i\le n-2$ and \[a_0 \stackrel{R}{\mbox{\textemdash}}  a_1 \stackrel{R}{\mbox{\textemdash}}  \cdots  \stackrel{R}{\mbox{\textemdash}}  a_{n-1}  \stackrel{R}{\mbox{\textemdash}}  a_n.\]

A reduced path  as above is \emph{closed} if $a_{n-1}$ is distinct from $a_0$ yet $ a_n \stackrel{R}{\mbox{\textemdash}}  a_0$. According to our  definitions, an edge $a \stackrel{R}{\mbox{\textemdash}}  b$ connecting two distinct points is a reduced path.

Note that if two points are connected, the reduced path connecting them is unique. The theory PS is $\omega$-stable, complete and model-complete \cite[Proposition 2.1]{BP00}. This theory does not have quantifier elimination, yet for certain subsets it is easy to determine whether they have the same type.  A non-empty subset $A$ of a model $M$ of PS is \emph{nice} if for any two given points in $A$, the (unique) reduced path connecting them lies in $A$ as well. The collection of $\LL$-isomorphism between nice subsets is a non-empty back-\&-forth system, so two nice subsets $A$ and $A'$ are $\LL$-isomorphic if and only if they have the same type.  
Moreover, nice sets are algebraically closed. The definable closure of a set $X$ is nice, coincides with the the algebraic closure and is obtained from $X$ by adding to it all possible reduced paths in $M$ between points in $X$.

\begin{lemma}\label{L:PS2}
Given sets $A, B$ and $C$ with $C\subset A\cap B$, non-forking independence in the stable theory PS can be characterised as follows: \\
\begin{center}
$A\ind_C B \  \Longleftrightarrow \ $ \begin{minipage}{10cm} whenever a point in $\acl(A)$ and  a point in $\acl(B)$ are connected, then the corresponding reduced path contains some point in $\acl(C)$.  \end{minipage}
\end{center}
\end{lemma}
\begin{proof}
Given $C\subset A\cap B$, set 
\begin{center}
$A\indone_C B \  \Longleftrightarrow $ \begin{minipage}{10cm} whenever a point in $\acl(A)$ and  a point in $\acl(B)$ are connected, then the corresponding reduced path contains some point in $\acl(C)$.   \end{minipage}
\end{center}

It is immediate to verify that the above notion of independence satisfies {\bf Invariance, Basedness, Symmetry} as well as {\bf Finite Character}.  {\bf Monotonicity} holds trivially. For {\bf Base Monotonicity}, 
assume that $C\subset B\subset D$ and  that the set $A$ containing $C$ satisfies that $A\indone_C D$. In order to show that $A\cup B \indone _B D$, consider a reduced path connecting a point $y$ in $\acl(D)$ with a point $x$ in $\acl(A\cup B)$. If $x$ belongs to $\acl(B)$, we are done for trivial reasons. If $x$ belongs to $\acl(A)$, then the reduced path connecting $x$ to $y$ contains a point in $\acl(C)\subset \acl(B)$, since $A\indone_C D$. Thus, we need only consider the case that $x$ in \[ \acl(A\cup B)\setminus  \left( \acl(A)\cup \acl(B) \right) \] appears as a vertex in a reduced path connecting an element $a$ in $A$ to some $b$ in $B$. Since $A\indone_C B$, an element $z$ in $\acl(C)$ occurs in the reduced path from $a$ to $b$. Note that $x\ne z$.  If $z$ occurs in the reduced subpath connecting $a$ to $x$, then  $x$ would be in the nice subset $\acl(B)$, for  both $z$ and $b$  belong to $\acl(B)$. If however the point $z$ of $\acl(C)$ occurs in the subpath from $x$ to $b$, then  $x$ belongs to $\acl(A)$,  since both $z$ and $a$ lie in $\acl(A)$, which gives the desired contradiction.  (This proof also yields that forking in the pseudoplane is trivial!).

For {\bf Transitivity}, assume that $C\subset B\subset D$ and  the set $A$ containing $C$ is such that $A\indone_C B$ and $A\cup B\indone_B D$. In order to check that $A\indone_C D$, assume that some point $x$ in $\acl(A)$ is connected to some point $y$ in $\acl(D)$. The reduced path connecting $x$ to $y$ must contain a point $z$ in $\acl(B)$, for $A\cup B \ind_B D$. Now, the subpath from $x$ to $z$ is also reduced,  so it must contain a point $u$ in $\acl(C)$, since $A\ind_C B$. The point $u$ lies in the reduced path from $x$ to $y$, so $A\indone_C D$, as desired.

\begin{claimstar}
The independence relation determined by $\indone$ is a weak  notion of independence.
\end{claimstar} 

\begin{claimstarproof}
Consider a tuple $a$ and an elementary substructure $M$ contained in the subset $B$ with $a\ind_M B$. We may clearly assume that $B=\acl(B)$. In order to show that $M\cup a\indone_M B$, consider some element $u$ in $\acl(M\cup a)$ which is connected to a point $b$ in $B$ by a reduced path of length $n$. If $u$ belongs to $M$, we are done. Otherwise, non-forking independence yields that $u\ind_M B$, so the type $\tp(u/B)$ is a heir of $\tp(u/M)$ in the stable theory PS. Hence, there is some point $m$ in $M$ connected to $u$ by a reduced path. We may assume that $m$ in $M$ is chosen in such a way that its distance $k\ge 1$ to $u$ is least possible.  Since  the elementary substructure $M$ is nice, it follows that  every reduced path from $u$ to the element $m'$ in $M$ decomposes as the path from $u$ to $m$ and then the reduced path in $M$ from $m$ to $m'$. 

It suffices to show that $m$ occurs in the path from $u$ to $b$. Assuming otherwise, there is a formula $\varphi(x, b, m)$ in $\tp(u/B)$ witnessing  that the element $m$ does not occur in the unique reduced path from $u$ to $b$. By the heir property, there is some  element $m'$ in $M$ such that $\varphi(u, m', m)$ holds, which gives the desired contradiction, since the vertex $m$ must occur in the reduced path from $u$ to $m'$.  
\end{claimstarproof}

Given nice subsets $C\subset B$ , we may use the fact that every point has infinitely many direct neighbors in order to produce a nice set $M$ containing $C$ with $M\indone_C B$ such that every point in $M$ has infinitely many neighbors in $M$. It then follows that $M$ is is a model of the model-complete theory PS \cite[Proposition 2.1]{BP00} and hence $M$ is an elementary substructure.   If we show that the weak notion of independence $\indone$ satisfies {\bf Stationarity over models}, then Remark \ref{R:weak_hyp} and the above discussion yields {\bf Weak Extension}, so 
 $\indone$ coincides with non-forking independence (over algebraically closed sets) by Proposition \ref{P:uniq}.  

 Assume therefore that the subsets $A$ and $A'$ have the same type over the elementary substructure $M$ contained in $B$ with $M\cup A\indone_M B$ and $M\cup A'\indone_M B$. 
Now, the isomorphism $F:M\cup A \to M\cup A'$ fixing $M$ pointwise extends to an $M$-isomorphism between the definable closures $\dcl(M\cup A)$ and $\dcl(M\cup A')$, which are both nice sets.  Since nice sets are algebraically closed, we deduce that \[ \dcl(M\cup A )\indone_M \acl(B) \text{ and } \dcl(M\cup A' )\indone_M \acl(B).\] It follows immediately that $\acl(B) \cup\dcl(M\cup A)$ and $\acl(B) \cup\dcl(M\cup A')$ are nice sets. The isomorphism $F$ extends to an isomorphism \[ \acl(B) \cup\dcl(M\cup A) \to \acl(B) \cup\dcl(M\cup A')\] fixing $\acl(B)$ pointwise, so $A\equiv_B A'$, as desired. 

\end{proof}

\subsection*{Differentially closed fields of characteristic $0$}
We fix some natural number $m\ge 1$.  The  theory DCF$_{0,m}$ of differentially closed fields of characteristic $0$ with
$m$ commuting derivations  \cite{tMG00} (see  \cite[Chapter II]{MMP96} for the particular case of $m=1$) is complete with quantifier elimination and  $\omega$-stable. A straightforward back-\&-forth argument yields that any two isomorphic differential subfields have the same type. In particular, the algebraic closure of a 
subset $A$ is obtained by taking the field algebraic closure of the differential field $\mathbb{Q}\!<\hskip-1mm A\hskip-1mm>$ generated by $A$. 

\begin{lemma}\label{L:DCF_indep}
Given  differential subfields $k\subset L\cap M$ of a model $K$ of DCF$_{0,m}$, we have the following characterization of non-forking independence:
\[ L\ind_k M \hskip5mm  \Longleftrightarrow \hskip5mm  L\cdot k^\text{alg} \ind^\textrm{ld}_{k^\text{alg}} M\cdot k^\text{alg},\] where $\ind^\textrm{ld}$ denotes linear disjointness (see Definition \ref{D:lindisj}) and $L\cdot k^\text{alg}$ is the differential field obtained as a compositum of $L$ and $k^\text{alg}$. 
\end{lemma}
Since the base field $k^\text{alg}$ is algebraically closed, Remark \ref{R:lindisj} yields that non-forking is equivalent to algebraic independence of $L$ and $M$ over $k$. The above result is well-known and we do not claim that it is particularly original. We only want to emphasize that Proposition \ref{P:uniq} yields a faster characterization of non-forking independence without having to use the full-strength of the work of Kim and Pillay.  
\begin{proof}
Given $C\subset A\cap B$, set \[ A\indone_C B \text { if andd only if } k^\text{alg}\!<\hskip-1mm A\hskip-1mm> \ind^\textrm{ld}_{k^\text{alg}} L\cdot k^\text{alg},\] where $k=\mathbb Q<\hskip-1mm C\hskip-1mm>$ is the differential field generated by $C$.  This notion of independence $\indone$ clearly satisfies  {\bf Invariance} and {\bf Basedness}, as well as {\bf Symmetry, Finite Character} and {\bf Monotonicity}, for these properties are true for linear disjointness, see Example \ref{E:alg_ind}.

For {\bf Base Monotonicity}, if $A\indone_C D$ and $C\subset B\subset D$, we have that the fields $k^\text{alg}\!<\hskip-1mm A\hskip-1mm>$ and $k^\text{alg}\!<\hskip-1mm D\hskip-1mm>$ are algebraically independent over $k^\text{alg}$ and thus so are their algebraic closures. Since the algebraic closure of $L=\mathbb Q<\hskip-1mm B\hskip-1mm>$ is contained in $k\!<\hskip-1mm D\hskip-1mm>^\text{alg}$, we have that \[ L^\text{alg}\cdot k^\text{alg}\!<\hskip-1mm A\hskip-1mm> \ind^\textrm{ld}_{L^\text{alg}} L^\text{alg}\!<\hskip-1mm D\hskip-1mm>,\] so by Remarks \ref{R:lindisj} and  \ref{R:Example_weight} (using that the derivation is \emph{operator-like}), we deduce that $B\cup A\ind_B D$, as desired. 

{\bf Transitivity} follows easily from the transitivity of linear disjointness. 
The proof of Lemma \ref{L:ld} yields that $\indone$ is a weak notion of independence, which clearly satisfies {\bf Stationarity over models}, since elementary substructures are in particular  algebraically closed fields. Moreover, whenever  two differential fields are linearly disjoint over a common differential subfield, their compositum can be equipped in a unique way with a derivation extending each of the corresponding derivations, so stationarity  follows immediately from the above description of types.  

{\bf Weak Extension} follows by an easy compactness argument using the fact that models of DCF$_{0,m}$ are existentially closed differential fields, so Proposition \ref{P:uniq} yields that $\indone$ coincides with non-forking independence (over algebraically closed subsets). 
\end{proof}

\subsection*{Separably closed fields of infinite imperfection degree}
 We refer to  \cite{gS86, fD88} for a detailed introduction to the model theory of separably closed fields.  Fix now an ambient field  of positive characteristic $p>0$, in which all sets and tuples are contained. Note that for any field $K$,  the $p^\text{th}$-powers $K^p=\{x^p\}_{x\in K}$ form a subfield of $K$, so it makes sense to consider the \emph{degree of imperfection} of  $K$ as the unique element $e$ of $\N\cup\{\infty\}$ with $[K:K^p]=p^e$. 
 
A (possibly non-algebraic) field extension $k\subset K$ is \emph{separable} if $ k\ld_{k^p} K^p$.  If $k\subset K$ is separable, a subset $B$ of $K$ is \emph{$p$-independent over $k$} if  no $b$ in $B$ belongs to $K^p[k\cup B\setminus \{b\}]$. From now on, we will say that $A$ is $p$-independent if $A$ is $p$-independent over $\Fp$.  Notice that, whenever $k\subset K$ is separable, a subset $A$ of $k$ is $p$-independent  if and only if $A$ is $p$-independent, seen as a subset of $K$. A \emph{$p$-basis} is a maximal $p$-independent subset. In particular, an extension $k\subset K$ is separable if and only if some (or equivalently, every) $p$-basis of $k$ remains $p$-independent in $K$, and thus, it extends to a $p$-basis of $K$.

\begin{fact}\label{F:p-base}
	Given a separable field extension $k\subset K$ as well as a $p$-independent subset $A\subset K$ over $k$ and an element $x$ of $K$ separably algebraic over $k$, the extensions $k(A)\subset K$ and $k(x)\subset K$ are again separable. 
\end{fact}

Ershov \cite{yE68} proved that the completions of the theory of separably closed fields of positive characteristic $p$ are uniquely determined by the imperfection degree.  We will denote by SCF$_{p,\infty}$ the complete theory of separably closed fields of characteristic $p>0$ and infinite imperfection degree. The models of SCF$_{p,\infty}$ are exactly the separably closed fields containing an infinite $p$-independent subset.  Ershov's proof that  the theory SCF$_{p,\infty}$ is complete and model-complete follows from the following result of Mac Lane: 

\begin{fact}\label{F:MacLane}\textup{(}\cf \cite[Theorem 15]{sML39}\textup{)} 
Consider a finitely generated separable field  extension $k\subset L$. For any $p$-basis $A$ of $L$ over $k$, the elements of $A$ are transcendental and algebraically independent over $k$. Moreover, the field extension $k(A)\subset L$ is separably algebraic. 
\end{fact}

Using a straight-forward compactness argument and Mac Lane's result, it is easy to show that two subfields $k$ and $k'$ of  a model $K$ of $SCF_{p,\infty}$ have the same type whenever $k$ and $k'$ are isomorphic and both field extensions $k\subseteq K$ and $k'\subseteq K$ are separable.  In particular, the theory  SCF$_{p,\infty}$  is stable. Moreover, if $k\subseteq K$ is separable, then $k$ is definably closed and its algebraic closure coincides with $k^{sep}$, the separable closure of $k$. The definable closure $\dcl_\lambda(A)$ of a subset $A$ of $K$ is the smallest subfield $k$ containing $A$ such that the field extension $k\subset K$ is separable (the field $k$ can be explicitly presented, adding  all iterated values of the generalized $\lambda$-functions starting from tuples in $A$. We will not use the generalized $\lambda$-functions in this note, so we will not introduce them formally for the sake of brevity).

\begin{lemma}\label{L:SCF}\textup{(}\cf \cite[Theorem 13]{gS86}\textup{)} 
Given subsets $A$, $B$ and $C=\acl_\lambda(C)\subset A\cap B$ in a model $K$ of SCF$_{p,\infty}$,   non-forking independence can be characterized as follows:

 \[ \begin{array}{rcl}A\ind\limits_{C} B & \mbox{ if and only if }
	& \left\{\begin{array}{c} \dcl_\lambda(A) \ind\limits^\textrm{ld}_{C} \dcl_\lambda(B)
\\[5mm] \mbox{and}\\[3mm] \text{the extension } \dcl_\lambda(A) \cdot \dcl_\lambda(B) \subset K \text{ is separable}
	\end{array}\right.
\end{array},\]
where $ \dcl_\lambda(A) \cdot \dcl_\lambda(B)$ denotes the compositum field of $ \dcl_\lambda(A)$ and $ \dcl_\lambda(B)$. 
\end{lemma}

\begin{proof}
As before, whenever $C\subset A\cap B$ are subsets of an ambient model $K$ of SCF$_{p,\infty}$, set $A\indone_C B$ if 
\begin{itemize}
\item the fields $k^\text{sep}\cdot \dcl_\lambda(A)$ and $k^\text{sep}\cdot \dcl_\lambda(B)$ are linearly disjoint over the separable closure $k^\text{sep}$ of $k=\dcl_\lambda(C)$ and;
\item some (or equivalently, every) $p$-basis of $ \dcl_\lambda(A)$ over $k$ remains $p$-independent over $\dcl_\lambda(B)$, which is equivalent by Fact \ref{F:p-base} and the previous independence, to requiring that  the field extension $k^\text{sep}\cdot \dcl_\lambda(A) \cdot \dcl_\lambda(B) \subset K$ is separable. 
\end{itemize}
By definition of $\indone$, {\bf Basedness} holds trivially. 
Example \ref{E:alg_ind}, Remark \ref{R:Example_weight} and Lemma \ref{L:ld} yield that $\indone$ defines a weak notion of independence satisfying {\bf Invariance,  Symmetry} and {\bf Finite Character}. {\bf Monotonicity} and {\bf Transitivity} follow analogously as in the proof of \ref{L:DCF_indep}, using that if $k\subset L$ are both subsfields of our model $K$ of SCF$_{p,\infty}$ with $k$ algebraically closed (in the model-theoretic sense), the field extension $k\subset L$ is regular \cite[Remarque 1.7]{BMP19}. 

Let us now show that $\indone$ satisfies {\bf Base Monotonicity}: Assume that $A\indone_C D$ with $C=\acl(C)\subset B\subset D$. The extension $\dcl_\lambda(A)\cdot \dcl_\lambda(D)\subset K$ is separable and thus so is $\dcl_\lambda(A)\cdot \acl_\lambda(D) \subset K$ by Fact \ref{F:p-base}.

Hence, every subset of $\dcl_\lambda(A)$ which is $p$-independent over $C$ remains so over $\acl_\lambda(D)$, and thus over $\acl_\lambda(B)=\dcl_\lambda(B)^{sep}$ by Fact \ref{F:MacLane}, so $A\indone_C B$ and hence \[ \dcl_\lambda(A\cup \acl_\lambda(B))=\dcl_\lambda(A)\cdot \acl_\lambda(B),\] so  $\dcl_\lambda(A) \ind^\textrm{ld}_{\dcl_\lambda(C)} \acl(D)$ implies that \[ \dcl_\lambda(A\cup\acl_\lambda(B)) \ind^\textrm{ld}_{\acl_\lambda(B)} \dcl_\lambda(D)\cdot \acl_\lambda(B). \] Hence, we deduce that $B\cup A\indone_{B} D$, as desired.  

{\bf Stationarity over models} follows now immediately from the aforementioned description of types. Indeed, if $M\subset B$ is an elementary substructure, given $A\equiv_M A'$ with both $M\cup A\indone_M B$ and $M\cup A'\indone_M B$, then the field isomorphism from $\dcl_\lambda(M\cup A)$ to $\dcl_\lambda(M \cup A')$ fixing $M$ pointwise extends by linear disjointness to a field isomorphism from $\dcl_\lambda(M\cup A)\cdot \dcl_\lambda(B)$ to $\dcl_\lambda(M\cup A')\cdot \dcl_\lambda(B)$ fixing $\dcl_\lambda(B)$ pointwise.  As the field extensions \[ \dcl_\lambda(M\cup A)\cdot \dcl_\lambda(B)\subset K \text{ and } \dcl_\lambda(M\cup A'
)\cdot \dcl_\lambda(B)\subset K\] are both separable, we conclude that $A\equiv_B A'$ as desired. 

Though Fact \ref{F:MacLane} allows to  prove {\bf Weak extension} directly, we will show {\bf Weak extension} using the following claim and  Remark \ref{R:weak_hyp}.

\begin{claimstar}
	For every algebraically closed subset $C=\acl_\lambda(C)\subset B$, there is some elementary substructure $M$ of $K$  containing $C$ with $M\indone_C B$. 
\end{claimstar} 

\begin{claimstarproof}
Given $C=\acl_\lambda(C)\subset B$, consider an infinite subset $A$ of $K$ consisting of elements $p$-independent over  $\dcl_\lambda(B)$ (we are using that the imperfection degree of $K$ is infinite). By Fact \ref{F:p-base}, both field extensions \[ C(A)\subset K \text{ and  } C(A)\cdot \dcl_\lambda(B)\subset K\] are separable.  Moreover, the elements of $A$ are transcendental over $\dcl_\lambda(B)$ by Fact \ref{F:MacLane}, and thus $\Fp(A)$ and $\dcl_\lambda(B)$ are linearly disjoint over the prime field $\Fp$. Hence, the fields  $C(A)$ and $\dcl_\lambda(B)$ are linearly independent over $C$ by base monotonicity of linear disjointness, and hence the fields $C(A)$ and $\dcl_\lambda(B)$ are algebraically  independendent over $C$ by Remark Remark \ref{R:lindisj}. 

Now set $M= C(A)^{sep}\subset K$, which is clearly a model of SCF$_{p,\infty}$, and thus an an elementary substructure containing $C$ by model-completeness. The elements of $M$ are  separably algebraic over $C(A)$, so $M$ and $\dcl_\lambda(B)$ are algebraically independent over  $C$.  By Fact \ref{F:p-base}, the extension $M\cdot \dcl_\lambda(B)\subset K$ is again separable. 

Since $C$ is algebraically closed in $K$, the extension $C\subset M$ is regular (again by  \cite[Remarque 1.7]{BMP19}), so we deduce from Remark \ref{R:lindisj} that  $M\indone_C B$, as desired. 
\end{claimstarproof}
We conclude now from Proposition \ref{P:uniq} and Remark \ref{R:weak_hyp} that $\indone$ coincides with non-forking independence over algebraically closed subsets. 
\end{proof}
\begin{remark}
The proof of {\bf Stationarity over models} in Lemma \ref{L:SCF} only uses that the elementary substructure $M$ is relatively algebraically closed in the ambient field $K$, so linear disjointness and algebraic independence over $M$ coincide. 

It follows from the above proof that types in the theory SCF$_{p,\infty}$ over real algebraically closed subsets are stationary, despite the fact that this theory does not eliminate imaginaries. The stationarity of types over algebraically closed subsets has already been shown by Bartnick \cite{cB24} for a wider class of stable theories including SCF$_{p,\infty}$. 
\end{remark}

\subsection*{Differentially closed fields of positive characteristic $p>0$}
The model companion DCF$_p$ of the theory of differential fields of of positive characteristic $p>0$ was first considered by Wood \cite{cW73, cW76} and later on by Shelah \cite{sS73}. Whilst its axiomatization can be done in a similar way to the case of DCF$_0$ (or rather DCF$_{0,1}$), the theory DCF$_p$ is strictly stable, so not superstable. This theory shares many common traits with the theory SCF$_{p,\infty}$, though the presence of the derivation renders the issues of separability much simpler to treat. 

Each model of the complete theory DCF$_p$ is a differential field $(K, \delta)$ satisfying the axiom scheme given by Blum \cite{lB68} such that $K$ is \emph{differentially perfect}, that is, an element of $K$ is a constant (that is, its derivative is $0$) if and only if it belongs to the subfield $K^p$ of $p^\text{th}$ powers.  Now, a differential subfield $k$ is always linearly disjoint from the constants of $K$ over the constants of $k$, so it follows that $k\subset K$ is separable if and only if $k$ itself is differentially perfect. In particular, any two differentially perfect isomorphic differential subfields have the same type in DCF$_p$ and such a subfield is definably closed. The definably closure $\dcl_\mathrm{DCF_p}(A)$ of a subset $A$ is the smallest definably perfect differential subfield containing $A$. It can be explictly constructed by successively closing $A$ under the field operations, the derivation and the function $\begin{array}[t]{rccl} r:& K&\to& K \\[1mm] & x&\mapsto& \begin{cases}
x^{\frac{1}{p}}, \text{ if } \delta(x)=0 \\[2mm]
0, \text{ otherwise}
\end{cases}\end{array}.$

\noindent 
By Lemma \ref{L:ld}, the independence relation 
 \[ A\indone_C B \hskip5mm  \Longleftrightarrow \hskip5mm \dcl_\mathrm{DCF_p}(A) \ind^\textrm{ld}_{\dcl_\mathrm{DCF_p}(C)} \dcl_\mathrm{DCF_p}(B)\] is a weak notion of independence. This notion of independence satisfies {\bf Invariance, Basedness, Symmetry, Monotonicity, Transitivity} and {\bf Finite Character}.

 	 Shelah \cite[Theorem 9]{sS73} proved that, whenever $M\subset B$ is an elementary substructure of DCF$_p$ and $A\indone_M B$, the compositum field \[  \dcl_{\mathrm{DCF}_p}(M\cup A) \cdot \dcl_{\mathrm{DCF}_p}(B) \] is  again differentially perfect. This yields immediately {\bf Stationarity over models} by the above description of types. In particular, we obtain a characterization of non-forking independence over models in terms of $\indone$ by Proposition \ref{P:uniq}. 
 	 
 	 However, the above notion of independence $\indone$ does not satisfy {\bf Base Monotonicity}, so we cannot conclude from Proposition \ref{P:uniq} that $\indone$ coincides with non-forking independence. In a private discussion with Omar Le\'on S\'anchez, we obtained an easy counter-example, given by subsets which are algebraically independendent over the algebraic closure of the prime field yet they are not forking-independent. 
 	 
 \section*{Acknowledgments}
 	 We would like to thank Martin Ziegler for having used in previous occasions  the stationarity of types in stable theories in order to yield an explicit description of non-forking. We regret (and are somewhat ashamed) that it took us so long to realize how to formalize what he already did implicitly. 
 	 
 	 We would like to thank Charlotte Bartnick and Daniel Palac\'in for their suggestions and remarks of a previous version of this note.

\end{document}